\documentclass{amsart}

\usepackage[latin1]{inputenc}
\usepackage[english]{babel}
\usepackage{indentfirst}
\usepackage{amssymb}
\usepackage{amsthm}

\newcommand{\N}{\mathbb{N}}
\newcommand{\sub}{\subseteq}
\def\epsilon{\varepsilon}

\newtheorem{theo}{Theorem}[section]
\newtheorem{lem}[theo]{Lemma}
\newtheorem{pro}[theo]{Proposition}
\newtheorem{cor}[theo]{Corollary}
\newtheorem{defi}[theo]{Definition}
\newtheorem{rem}[theo]{Remark}

\numberwithin{equation}{section}

\title{A Diestel-Faires type result for multimeasures}

\author{Jos\'{e} Rodr\'{i}guez}

\address{Dpto. de Matem\'{a}ticas\\E.T.S. de Ingenier\'{i}a Agron\'{o}mica y de Montes y Biotecnolog\'{i}a\\
Universidad de Castilla-La Mancha\\ 02071 Albacete\\ Spain} 
\email{jose.rodriguezruiz@uclm.es}



\thanks{The research was supported by grants PID2021-122126NB-C32 
(funded by MCIN/AEI/10.13039/501100011033 and ``ERDF A way of making Europe'', EU) and 
21955/PI/22 (funded by {\em Fundaci\'on S\'eneca - ACyT Regi\'{o}n de Murcia}).}

\begin{document}

\begin{abstract}
Let $X$ be a real Banach space and let $Y \sub X^*$ be a linear subspace having the Orlicz-Thomas property, that is, 
for each $\sigma$-algebra $\Sigma$ and for each map $\nu:\Sigma\to X$, the countable additivity
of the composition $x^*\circ \nu$ for all $x^*\in Y$ implies the countable additivity of~$\nu$. 
We show that the Orlicz-Thomas
property allows to test countable additivity of set-valued maps. Namely, if $M$ is a map defined on a $\sigma$-algebra~$\Sigma$ whose values are 
convex, $\sigma(X,Y)$-compact, bounded non-empty subsets of~$X$, then the following statements are equivalent:
\begin{enumerate}
\item[(i)] $M$ is a strong multimeasure, that is, for every disjoint sequence $(A_n)_{n}$ in~$\Sigma$ the series of sets
$\sum_n M(A_n)$ is unconditionally convergent and the equality $M(\bigcup_n A_n)=\sum_n M(A_n)$ holds.
\item[(ii)] $M$ is a multimeasure, that is, for every $x^*\in X^*$ the support map $s(x^*,M):\Sigma \to \mathbb{R}$ 
defined by $s(x^*,M)(A):=\sup \{x^*(x):x\in M(A)\}$ is countably additive.
\item[(iii)] $s(x^*,M)$ is countably additive for every $x^*\in Y$.
\end{enumerate}
As an application, we give a result on the factorization of multimeasures through reflexive Banach spaces.
\end{abstract}

\maketitle

\section{Introduction}

Let $X$ be a real Banach space. The classical Orlicz-Pettis theorem states that a map $\nu:\Sigma\to X$, defined on a 
$\sigma$-algebra~$\Sigma$, is a countably additive vector measure if (and only if)
the composition $x^*\circ \nu$ is a countably additive scalar measure for every $x^*\in X^*$ (the dual of~$X$); see, e.g., \cite[p.~22, Corollary~4]{die-uhl-J}. 
Diestel and Faires~\cite{die-fai} (cf. \cite[p.~23, Corollary~7]{die-uhl-J}) showed that,
when $X$ contains no subspace isomorphic to~$\ell_\infty$, in order to get
the countable addivity of~$\nu$ it is enough to test the countable additivity of $x^*\circ \nu$ for every $x^*$ belonging to a {\em total} subset
of~$X^*$ (that is, a subset which separates the points of~$X$). A set $W \sub X^*$ is said to have the
{\em Orlicz-Thomas property} if a map $\nu:\Sigma \to X$, defined on a $\sigma$-algebra~$\Sigma$, 
is countably additive whenever $x^*\circ \nu$ is countably additive for every $x^*\in W$. This concept goes
back to~\cite{tho} and has been recently studied in~\cite{oka-rod-san}. A set having the Orlicz-Thomas property is total,
while the converse holds when $X$ contains no subspace isomorphic to~$\ell_\infty$, thanks to
the aforementioned result of Diestel and Faires.
Without additional assumptions on~$X$, it is known that every James boundary of~$X$ has the Orlicz-Thomas property, see
\cite[Proposition~2.9]{fer-alt-4} (cf. \cite[Remark~3.2(i)]{nyg-rod}).

In this paper we address similar questions for \emph{set-valued} maps, that is, whose values are subsets of a Banach space. 
There are different concepts of countable additivity for set-valued maps (see, e.g., \cite{hes-J} and the references therein). To deal with them 
we first need to introduce some terminology.
Given a sequence $(C_n)_{n}$ of non-empty subsets of~$X$, the series $\sum_{n} C_n$ is said to be {\em unconditionally convergent}
if the series $\sum_{n} x_n$ is unconditionally convergent in~$X$ for every sequence $(x_n)_{n}$ such that $x_n\in C_n$ for all $n\in \N$;
in this case, we define the \emph{sum} of the series by
$$
	\sum_{n}C_n:=\left\{\sum_{n}x_n:\ x_n\in C_n \text{ for all $n\in \N$}\right\}.
$$
In the set-valued setting, the role of the composition with elements of~$X^*$ is played by the so-called {\em support maps}.
We denote
\begin{eqnarray*}
	b(X) &:=&\{C \sub X: \, C \text{ is bounded and non-empty}\}, \\
	cwk(X)&:=&\{C \sub X: \, C \text{ is convex, weakly compact and non-empty}\}.
\end{eqnarray*}
Given $C \in b(X)$ and $x^*\in X^*$, let
$$
	s(x^*,C):=\sup\{x^*(x): \, x\in C\}.
$$
Given a set-valued map $M:\Sigma \to b(X)$ defined on a $\sigma$-algebra~$\Sigma$ and $x^*\in X^*$, the \emph{support map}
$s(x^*,M): \Sigma \to \mathbb{R}$ is defined by 
$$
	s(x^*,M)(A):=s(x^*,M(A))
	\quad \text{for all $A\in \Sigma$}.
$$
Observe that we commit a rather usual abuse of notation here.

\begin{defi}\label{defi:multimeasure}
Let $X$ be a Banach space and let $M:\Sigma \to b(X)$ be a set-valued map defined on a $\sigma$-algebra~$\Sigma$.
We say that:
\begin{enumerate}
\item[(i)] $M$ is a {\em strong multimeasure} if for every disjoint sequence $(A_n)_{n}$ in~$\Sigma$ the series
$\sum_{n} M(A_n)$ is unconditionally convergent and 
$$
	\overline{M\left(\bigcup_{n} A_n\right)}=\overline{\sum_{n} M(A_n)}
$$
(the closures being taken in the norm topology).
\item[(ii)] $M$ is a {\em multimeasure} if $s(x^*,M)$ is countably additive for every $x^*\in X^*$.
\end{enumerate}
\end{defi}

It is not difficult to check that every strong multimeasure is a multimeasure. 
Conversely, if a multimeasure takes values in~$cwk(X)$, 
then $M$ is a strong multimeasure (cf. \cite[Theorem~3.4]{cas-kad-rod-2}). 
This can be seen as a set-valued version of the Orlicz-Pettis theorem and raises the question of whether
an analogue of the Diestel-Faires theorem holds in this context. Our main result states that this is
indeed the case. More precisely, in Theorem~\ref{theo:OT-2} we prove that a set-valued map 
$M:\Sigma \to b(X)$, defined on a $\sigma$-algebra~$\Sigma$, is a
strong multimeasure provided that there is a linear subspace $Y \sub X^*$ having the Orlicz-Thomas property such that:
(i)~$M(A)$ is convex and $\sigma(X,Y)$-compact for every $A\in \Sigma$; and (ii)~$s(x^*,M)$ is countably additive for every $x^*\in Y$.
We stress that Musia\l \ proved in \cite[Theorem~3.4]{mus13} the particular case of Theorem~\ref{theo:OT-2}
when $X=Y^*$ is the dual of a Banach space $Y$  and $X$ contains no subspace isomorphic to~$\ell_\infty$. 
From the technical point of view, our approach to Theorem~\ref{theo:OT-2} uses
some ideas of the proof of \cite[Theorem~3.4]{cas-kad-rod-2} and a selection result for multimeasures in 
locally convex spaces which can be obtained as a consequence of results from the 1970's by different authors (Cost\'{e}, Drewnowski,
Godet-Thobie, Pallu de la Barri\`{e}re), see \cite{hes-J} and the references therein. In the Appendix we provide an accesible proof of that result for the reader's convenience. 

The paper is organized as follows. In Section~\ref{section:preliminaries} we collect
some preliminary known facts. Section~\ref{section:main} is devoted to proving our main result, Theorem~\ref{theo:OT-2}. 
Finally, in Section~\ref{section:applications} we provide an application to the 
factorization of multimeasures. It is known that the Davis-Figiel-Johnson-Pe{\l}czy\'nski factorization procedure
can be used to prove that
every countably additive vector measure $\nu:\Sigma\to X$, defined on a $\sigma$-algebra~$\Sigma$, factors through a reflexive Banach space~$Z$, in the sense
that there exist a countably additive vector measure $\tilde{\nu}:\Sigma\to Z$ and a one-to-one, bounded, linear operator
$T:Z\to X$ such that $\nu=T\circ\tilde{\nu}$ (see \cite[Theorem~2.1(i)]{rod15}, cf. \cite[Theorem~4.1]{nyg-rod}). Here we show that an analogue statement holds
for $cwk(X)$-valued multimeasures (Theorem~\ref{theo:DFJP}).

\section{Notation and preliminaries}\label{section:preliminaries}

Throughout this paper we deal with real linear spaces only. Let $E$ be a linear space. The convex hull of a set $C \sub E$
is denoted by ${\rm conv}(C)$. Given finitely many non-empty subsets of~$E$, say
$C_1,\dots,C_p$, its {\em sum} is defined as 
$$
	\sum_{i=1}^p C_i:=\left\{\sum_{i=1}^p x_i: \, x_i \in C_i \, \text{ for every $i\in \{1,\dots,p\}$}\right\}.
$$

Let $X$ be a Banach space. The evaluation of $x^*\in X^*$ at $x\in X$ is denoted by either
$x^*(x)$ or $\langle x^*,x\rangle$. The closed unit ball of~$X$ is denoted by~$B_X$. Given a bounded non-empty set $C$, we write
$\|C\|:=\sup\{\|x\|:x\in C\}$, where $\|\cdot\|$ stands for the norm of~$X$. 

The following lemma is a well known characterization of unconditional convergence for series of sets (see, e.g., \cite[p.~544]{cas-rod-2} and the
references therein).

\begin{lem}\label{lem:ucsets-1}
Let $X$ be a Banach space and let $(C_n)_n$ be a sequence of non-empty subsets of~$X$. The following statements are equivalent:
\begin{enumerate}
\item[(i)] $\sum_n C_n$ is unconditionally convergent.
\item[(ii)] For every $\epsilon>0$
there is $N_\epsilon \in \N$ such that $\|\sum_{n\in P} C_n\| \leq \epsilon$
for every finite non-empty set $P \sub \N \setminus \{1,\dots,N_\epsilon\}$.
\end{enumerate}
If each $C_n$ is bounded and the series $\sum_n C_n$ is unconditionally convergent, then the set $\sum_n C_n$ is bounded and, for each $x^*\in X^*$, the series 
$\sum_n s(x^*,C_n)$ is absolutely convergent, with sum $s(x^*,\sum_n C_n)$.
\end{lem}

The sum of an unconditionally convergent series of weakly compact sets is
weakly compact, see \cite[Lemma~2.2]{cas-rod-2}. We next extend that result to $\sigma(X,Y)$-compact sets, where $X$ is a Banach space,
$Y\sub X^*$ is a total set and $\sigma(X,Y)$ denotes the (locally convex Hausdorff) topology on~$X$ of pointwise convergence on~$Y$.

\begin{lem}\label{lem:compactness}
Let $X$ be a Banach space and let $Y \sub X^*$ be a total set. Let $(C_n)_n$ be a sequence of $\sigma(X,Y)$-compact non-empty subsets of~$X$. If
the series $\sum_n C_n$ is unconditionally convergent, then its sum $\sum_n C_n$ is $\sigma(X,Y)$-compact.
\end{lem}
\begin{proof}
Observe that $K:=\prod_n C_n$ is a compact topological space when equipped with the product topology (denoted by~$\mathfrak{T}$) and 
each $C_n$ is equipped with the relative $\sigma(X,Y)$-topology. For each $n\in \N$, we denote by $\pi_n: K \to C_n$
the $n$th-coordinate projection. Since the series $\sum_n C_n$ is unconditionally convergent,
we can define a map $S: K \to X$ by $S(z):=\sum_n \pi_n(z)$ for all $z\in K$, so that $S(K)=\sum_n C_n$. 

We will prove that $S$ is $\mathfrak{T}$-$\sigma(X,Y)$-continuous. Take a net $(z_\alpha)_{\alpha}$ in~$K$ which $\mathfrak{T}$-converges to some 
$z \in K$. Fix $x^*\in Y$ and $\epsilon>0$.
By Lemma~\ref{lem:ucsets-1}, there is $N_\epsilon\in \N$ such that $\|\sum_{n\in P} C_n\|\leq \epsilon$ for every finite non-empty set $P\sub \N$
disjoint from $\{1,\dots,N_\epsilon\}$. This yields
\begin{eqnarray*}
	|\langle x^*,S(z_\alpha)-S(z)\rangle| 
	&\leq& \sum_{n=1}^{N_\epsilon} \left|\left\langle x^*, \pi_n(z_\alpha)-\pi_n(z)\right\rangle\right| \\
	& & +\left|\left\langle x^*,\sum_{n>N_\epsilon} \pi_n(z_\alpha)\right\rangle\right|
	+\left|\left\langle x^*,\sum_{n>N_\epsilon} \pi_n(z)\right\rangle\right|
	\\ &\leq&
	\sum_{n=1}^{N_\epsilon} \left|\left\langle x^*, \pi_n(z_\alpha)-\pi_n(z)\right\rangle\right|+2\|x^*\|\epsilon
\end{eqnarray*}
for every $\alpha$. Since for each $n\in \{1,\dots,N_\epsilon\}$ the net $(\pi_n(z_\alpha))_{\alpha}$ is $\sigma(X,Y)$-convergent to~$\pi_n(z)$,
there is $\alpha_\epsilon$ such that $|\langle x^*,S(z_\alpha)-S(z)\rangle| \leq (1+2\|x^*\|)\epsilon$ for every $\alpha \geq \alpha_\epsilon$.
As $x^*\in Y$ and $\epsilon>0$ are arbitrary, the net $(S(z_\alpha))_{\alpha}$ is $\sigma(X,Y)$-convergent to~$S(z)$. 

Hence, $S$ is $\mathfrak{T}$-$\sigma(X,Y)$-continuous and so the set $S(K)=\sum_n C_n$ is $\sigma(X,Y)$-compact.
\end{proof}

\begin{cor}\label{cor:strong-multimeasure-compact}
Let $X$ be a Banach space, let $Y \sub X^*$ be a total set and
let $M:\Sigma \to b(X)$ be a set-valued map defined on a $\sigma$-algebra~$\Sigma$ 
such that $M(A)$ is $\sigma(X,Y)$-compact for every $A\in \Sigma$. The following statements are equivalent:
\begin{enumerate}
\item[(i)] $M$ is a strong multimeasure.
\item[(ii)] For every disjoint sequence $(A_n)_n$ in~$\Sigma$ the series
$\sum_n M(A_n)$ is unconditionally convergent and 
$$
	M\left(\bigcup_n A_n\right)=\sum_n M(A_n).
$$
\end{enumerate}
\end{cor}
\begin{proof}
Every $\sigma(X,Y)$-compact set is $\sigma(X,Y)$-closed and so norm closed.
The conclusion follows from Lemma~\ref{lem:compactness} and the very definition of strong multimeasure.
\end{proof}

\section{Main result}\label{section:main}

This section is devoted to proving our main result, Theorem~\ref{theo:OT-2} below.
We first isolate for future reference the following elementary observation:

\begin{lem}\label{lem:inf}
Let $C$ be a non-empty set, let $\varphi:C \to \mathbb{R}$ be a bounded function and define
$$
	s(\varphi,C):=\sup\{\varphi(x): \, x\in C\}.
$$
Then for every $x_0\in C$ we have
\begin{itemize}
\item $-s(-\varphi,C)=\inf\{\varphi(x): \, x\in C\}\leq \varphi(x_0)$;
\item $|\varphi(x_0)|\leq \max\{|s(\varphi,C)|,|s(-\varphi,C)|\}\leq |s(\varphi,C)| + |s(-\varphi,C)|$.
\end{itemize}
\end{lem}

\begin{theo}\label{theo:OT-2}
Let $X$ be a Banach space and let $M:\Sigma \to b(X)$ be a set-valued map defined on a $\sigma$-algebra~$\Sigma$. 
Suppose that there is a linear subspace $Y \sub X^*$ having the Orlicz-Thomas property such that:
\begin{enumerate}
\item[(i)] $M(A)$ is convex and $\sigma(X,Y)$-compact for every $A\in \Sigma$;
\item[(ii)] $s(x^*,M)$ is countably additive for every $x^*\in Y$.
\end{enumerate}
Then $M$ is a strong multimeasure.
\end{theo}
\begin{proof}
We divide the proof into several steps. Note that
the dual of the locally convex space $(X,\sigma(X,Y))$ is~$Y$ (because $Y$ is a linear subspace).

{\sc Step~1.} We have
\begin{equation}\label{eqn:fa}
	M(A\cup B)=M(A)+M(B)
	\quad\text{whenever $A,B \in \Sigma$ are disjoint}.
\end{equation} 
Indeed, note that the set $M(A)+M(B)$ is $\sigma(X,Y)$-compact and that
$$
	s(x^*,M(A)+M(B))=s(x^*,M(A))+s(x^*,M(B))
	\quad \text{for all $x^*\in X^*$}.
$$
Now, for each $x^*\in Y$, the fact that $s(x^*,M)$ is finitely additive yields
$$
	s(x^*,M(A\cup B))=s(x^*,M(A))+s(x^*,M(B))=s(x^*,M(A)+M(B)).
$$
Since both $M(A\cup B)$ and $M(A)+M(B)$ are convex and $\sigma(X,Y)$-closed, equality~\eqref{eqn:fa}
follows from the Hahn-Banach theorem.

{\sc Step~2.} 
By Theorem~\ref{theo:ca-selector} in the Appendix applied to $M$ as a set-valued map whose values are convex compact non-empty subsets 
of $(X,\sigma(X,Y))$, there is a map $\nu:\Sigma \to X$ such that: 
\begin{itemize}
\item $x^*\circ \nu$ is countably additive for all $x^*\in Y$;
\item $\nu(A)\in M(A)$ for all $A\in \Sigma$.
\end{itemize}
Since $Y$ has the Orlicz-Thomas property, $\nu$ is a countably additive vector measure.

Define a set-valued map $N:\Sigma \to b(X)$ by
$$
	N(A):=M(A)+\{-\nu(A)\}
	\quad
	\text{for all $A\in \Sigma$}.
$$
Note that $N(A)$ is convex and $\sigma(X,Y)$-compact for every $A\in \Sigma$ and that
$$
	s(x^*,N)=s(x^*,M)-x^*\circ \nu
$$
is countably additive for every $x^*\in Y$. 
By {\sc Step~1} applied to~$N$ we have 
\begin{equation}\label{eqn:fa2}
	N(A\cup B)=N(A)+N(B)
	\quad\text{whenever $A,B \in \Sigma$ are disjoint}.
\end{equation}

Let $\Omega$ be the set on which $\Sigma$ is a $\sigma$-algebra.
Given an arbitrary $A\in \Sigma$, we have $\nu(\Omega \setminus A)\in M(\Omega \setminus A)$, so
$0\in N(\Omega \setminus A)$ and therefore
$$
	N(A) \sub N(A)+N(\Omega\setminus A) \stackrel{\eqref{eqn:fa2}}{=} N(\Omega).
$$ 
It follows that the set 
$$
	C:=\bigcup_{A\in \Sigma}N(A)
$$
is contained in $N(\Omega)$ and so it is relatively $\sigma(X,Y)$-compact.

{\sc Step~3.} Let $(A_n)_n$ be a disjoint sequence in~$\Sigma$. We claim that the series
$\sum_n M(A_n)$ is unconditionally convergent. Indeed, since $\nu$ is countably additive, it suffices
to check that $\sum_n N(A_n)$ is unconditionally convergent.
Take $x_n\in N(A_n)$ for all $n\in \N$. We will prove that $\sum_n x_n$ is unconditionally convergent.

Note first that, for each $x^*\in Y$, both $s(x^*,N)$ and $s(-x^*,N)$ are countably additive and so
$$
    \sum_{n}|x^*(x_{n})|\leq
    \sum_{n}|s(x^*,N(A_{n}))|+
    \sum_{n}|s(-x^*,N(A_{n}))|<\infty
$$
(use Lemma~\ref{lem:inf}). 
Write 
$$
	z_P:=\sum_{n\in P} x_n
$$ 
for every finite non-empty set~$P \sub \N$. 

Fix a non-empty set $S\sub \N$.
We will show that the net $\mathcal{N}_S:=(z_P)_P$, where $P$ runs over all finite non-empty subsets of~$S$ ordered by inclusion,
is $\sigma(X,Y)$-convergent. On the one hand,
for each $x^*\in Y$ the series $\sum_n x^*(x_n)$ is absolutely convergent and so
the net of real numbers $(x^*(z_P))_P$, where $P$ runs over all finite non-empty subsets of~$S$ ordered by inclusion,
is convergent, with limit $\sum_{n\in S} x^*(x_n)$. Therefore, the net $\mathcal{N}_S$ has at most one $\sigma(X,Y)$-cluster point in~$X$.
On the other hand, for every finite non-empty set $P \sub S$ we have
$$
	z_P\in \sum_{n\in P} N(A_{n}) \stackrel{\eqref{eqn:fa2}}{=} N\left(\bigcup_{n\in P} A_{n}\right) \sub C
$$
and $C$ is relatively $\sigma(X,Y)$-compact. It follows that the net $\mathcal{N}_S$ is $\sigma(X,Y)$-convergent to some $\mu(S)\in X$. 

Set $\mu(\emptyset):=0$ and consider the map $\mu:\mathcal{P}(\N)\to X$, where $\mathcal{P}(\N)$ is the power set of~$\N$. 
For each $x^*\in Y$ the composition $x^*\circ \mu$ is countably additive, because
the series $\sum_n x^*(x_n)$ is absolutely convergent and  
$$
	x^*(\mu(S))=\sum_{n\in S}x^*(x_n)
	\quad
	\text{for every non-empty set $S\sub \N$}.
$$	
Since $Y$ has the Orlicz-Thomas property, it follows that $\mu$
is a countably additive vector measure. In particular, the series
$\sum_{n}x_{n}=\sum_n\mu(\{n\})$ is unconditionally convergent in~$X$. 

{\sc Step~4.} Let $(A_n)_n$ be a disjoint sequence in~$\Sigma$. Then
$\sum_n M(A_n)$ is unconditionally convergent (by {\sc Step~3}).
The sets $M(\bigcup_n A_n)$ and $\sum_n M(A_n)$ are $\sigma(X,Y)$-compact (the latter by Lemma~\ref{lem:compactness}) and so 
$\sigma(X,Y)$-closed. Moreover, they are 
convex. Since
$$
	s\left(x^*,M\left(\bigcup_n A_n\right)\right)=
	\sum_n s(x^*,M(A_n))=s\left(x^*,\sum_n M(A_n)\right)
	\quad
	\text{for all $x^*\in Y$}
$$
(by  the countable additivity of each $s(x^*,M)$ and Lemma~\ref{lem:ucsets-1}), from the Hahn-Banach theorem we get
$$
	M\left(\bigcup_n A_n\right)=\sum_n M(A_n).
$$
This proves that $M$ is a strong multimeasure. The proof is finished.
\end{proof}

By putting together the Diestel-Faires theorem and Theorem~\ref{theo:OT-2}, we get:

\begin{cor}\label{cor:DF-1}
Let $X$ be a Banach space not containing subspaces isomorphic to~$\ell_\infty$ and 
let $M:\Sigma \to b(X)$ be a set-valued map defined on a $\sigma$-algebra~$\Sigma$. 
Suppose that there is a total linear subspace $Y \sub X^*$ such that:
\begin{enumerate}
\item[(i)] $M(A)$ is convex and $\sigma(X,Y)$-compact for every $A\in \Sigma$;
\item[(ii)] $s(x^*,M)$ is countably additive for every $x^*\in Y$.
\end{enumerate}
Then $M$ is a strong multimeasure.
\end{cor}

In particular, we obtain the following result proved in \cite[Theorem~3.4]{mus13}. 
Given a Banach space~$X$, we denote by $cw^*k(X^*)$ the family of all convex weak$^*$-compact non-empty 
subsets of~$X^*$.

\begin{cor}[Musia{\l}]\label{cor:DFdual}
Let $X$ be a Banach space such that $X^*$ does not contain subspaces isomorphic to~$\ell_\infty$. Let 
$M:\Sigma \to cw^*k(X^*)$ be a set-valued map defined on a $\sigma$-algebra~$\Sigma$ such that $s(x,M)$ is countably additive for every $x\in X$. 
Then $M$ is a strong multimeasure.
\end{cor}

\begin{cor}\label{cor:DF-2}
Let $X$ be a Banach space and let $Y \sub X^*$ be a linear subspace having the Orlicz-Thomas property. Let 
$M:\Sigma \to cwk(X)$ be a set-valued map defined on a $\sigma$-algebra~$\Sigma$ such that $s(x^*,M)$ is countably additive for every $x^*\in Y$.
Then $M$ is a strong multimeasure.
\end{cor}

\begin{rem}\label{rem:cases}
\rm 
\begin{enumerate}
\item[(i)] Let $X$ be a Banach space. As we already pointed out in the Introduction, a set $Y \sub X^*$ has the Orlicz-Thomas property
if it contains a James boundary of~$X$ (see \cite[Proposition~2.9]{fer-alt-4}, cf. \cite[Remark~3.2(i)]{nyg-rod}), that is, a set $B \sub B_{X^*}$
with the property that for every $x\in X$ there is $x^*\in B$ such that $\|x\|=x^*(x)$. We stress that in this case every bounded, $\sigma(X,Y)$-compact subset of~$X$
is weakly compact, according to a striking result of Pfitzner~\cite{pfi-J}.
\item[(ii)] An example is given by the Banach space $X=C(K)$ of all continuous real-valued functions on 
a compact Hausdorff topological space~$K$, equipped with the supremum norm, and an arbitrary
set $Y \sub C(K)^*$ containing $B:=\{\delta_t:t\in K\}$, where $\delta_t\in C(K)^*$ denotes the evaluation functional at the point~$t\in K$.
In this case, $\sigma(C(K),Y)$-compact subsets of~$X$ are pointwise compact and so, for bounded sets, $\sigma(C(K),Y)$-compactness
is equivalent to weak compactness (see, e.g., \cite[Theorem~3.139]{fab-ultimo}).
\end{enumerate}
\end{rem}

\section{Application to the factorization of multimeasures}\label{section:applications}

Recall that, for an arbitrary Banach space~$X$, a set-valued map defined on a $\sigma$-algebra and taking values in $cwk(X)$ is a multimeasure
if and only if is a strong multimeasure. This section is devoted to proving the following:

\begin{theo}\label{theo:DFJP}
Let $X$ be a Banach space and let $M:\Sigma\to cwk(X)$ be a multimeasure defined on a $\sigma$-algebra~$\Sigma$. Then there exist
a reflexive Banach space~$Z$, a one-to-one, bounded, linear operator $T:Z\to X$ and a multimeasure $\hat{M}:\Sigma \to cwk(Z)$ such that
$T(\hat{M}(A))=M(A)$ for all $A\in \Sigma$.
\end{theo}

Besides our previous work, the main tools used to get this result are the Davis-Figiel-Johnson-Pe{\l}czy\'nski factorization method
and Corollary~\ref{cor:weaklycompactrange} below, which is known (see \cite[Proposition~2.4]{kan}) and a particular case of the following theorem.

Let $X$ be a Banach space and let $Y \sub X^*$ be a total linear subspace. The Mackey topology $\mu(X,Y)$ is the 
(locally convex Hausdorff) topology on~$X$ of uniform convergence on absolutely convex weak$^*$-compact subsets of~$Y$.
The locally convex spaces $(X,\sigma(X,Y))$ and $(X,\mu(X,Y))$ have the same topological dual (namely, $Y$),
the same bounded sets and the same closed convex sets (see, e.g., \cite[Section~3.5]{fab-ultimo}).
When $Y=X^*$, $\mu(X,Y)$ is just the norm topology on~$X$. When $Y$ is norm closed, the locally convex space
$(X,\mu(X,Y))$ is complete if and only if it is quasi-complete (its closed bounded subsets are complete), see \cite[Proposition~2.6]{gui-mar-rod}.
The completeness of $(X,\mu(X,Y))$ has been studied in \cite{bon-cas,gui-mar-rod,gui-mon-ziz} among others. For instance, 
$(X,\mu(X,Y))$ is complete whenever $X$ is weakly compactly generated and $Y$ is norming and norm closed (see \cite[Theorem~4]{gui-mon-ziz}).

\begin{theo}\label{theo:range}
Let $X$ be a Banach space and let $Y \sub X^*$ be a total linear subspace such that $(X,\mu(X,Y))$ is quasi-complete.
Let $M:\Sigma\to b(X)$ be a set-valued map defined on a $\sigma$-algebra~$\Sigma$ such that:
\begin{enumerate}
\item[(i)] $M(A)$ is relatively $\sigma(X,Y)$-compact for every $A\in \Sigma$;
\item[(ii)] $s(x^*,M)$ is countably additive for every $x^*\in Y$.
\end{enumerate}
Then $\bigcup_{A\in \Sigma}M(A)$ is relatively $\sigma(X,Y)$-compact.
\end{theo}
\begin{proof} We will apply James' compactness theorem (see, e.g., \cite[6.8]{flo})
in the locally convex space $(X,\sigma(X,Y))$. Write $C:=\bigcup_{A\in \Sigma}M(A)$. For each $x^*\in Y$ we have
\begin{equation}\label{eqn:bounded}
	-s(-x^*,M(A))\leq x^*(x) \leq s(x^*,M(A))
	\quad\text{for all $A\in \Sigma$ and $x\in M(A)$}. 
\end{equation}
Since both $s(-x^*,M)$ and $s(x^*,M)$ are countably additive, they have bounded range and 
so~\eqref{eqn:bounded} implies that the set $\{x^*(x):x\in C\}$ is bounded. 
It follows that $C$ is $\sigma(X,Y)$-bounded and so $\overline{{\rm conv}(C)}^{\sigma(X,Y)}$ is $\mu(X,Y)$-complete. Now, according to James'
compactness theorem, in order to prove
that $\overline{C}^{\sigma(X,Y)}$ is $\sigma(X,Y)$-compact it suffices to check that
each $x^*\in Y$ attains its supremum on~$\overline{C}^{\sigma(X,Y)}$.

Fix $x^*\in Y$. Since $s(x^*,M)$ is countably additive,  the Hahn decomposition theorem ensures the existence of $H\in \Sigma$ such that 
\begin{itemize}
\item $s(x^*,M(B))\geq 0$ for every $B\in \Sigma$ with $B \sub H$; 
\item $s(x^*,M(B))\leq 0$ for every $B\in \Sigma$ with $B\sub \Omega \setminus H$.
\end{itemize}
Since $M(H)$ is relatively $\sigma(X,Y)$-compact, there is $x_0\in \overline{M(H)}^{\sigma(X,Y)} \sub \overline{C}^{\sigma(X,Y)}$ 
such that 
$$
	x^*(x_0)=s\left(x^*,\overline{M(H)}^{\sigma(X,Y)}\right)=s(x^*,M(H)).
$$ 
Then for every $A\in \Sigma$ we have
\begin{eqnarray*}
	s(x^*,M(A))&=&s(x^*,M(A\cap H))+s(x^*,M(A\setminus H)) \\
	&\leq &s(x^*,M(A\cap H)) \ \leq \ s(x^*,M(H)) \ = \ x^*(x_0).
\end{eqnarray*}
Hence, $x^*(x)\leq x^*(x_0)$ for all $x\in C$ and so $x^*(x)\leq x^*(x_0)$ for all $x\in \overline{C}^{\sigma(X,Y)}$. This shows that
$x^*$ attains it supremum on~$\overline{C}^{\sigma(X,Y)}$ at~$x_0$. The proof is finished.
\end{proof}

\begin{cor}\label{cor:weaklycompactrange}
Let $X$ be a Banach space and let $M$ be a multimeasure defined on a $\sigma$-algebra~$\Sigma$ such that $M(A)$ is a relatively weakly
compact non-empty subset of~$X$ for every $A\in \Sigma$. Then $\bigcup_{A\in \Sigma}M(A)$ is relatively weakly compact.
\end{cor}

Our next result gives a sufficient condition for a set-valued map to be a multimeasure without assuming that its values are convex.

\begin{pro}\label{pro:KS}
Under the assumptions of Theorem~\ref{theo:range}, if in addition $Y$ has the Orlicz-Thomas property and $s(x^*,M)$ is finitely additive for every $x^*\in X^*$,
then $M$ is a multimeasure.
\end{pro}
\begin{proof}
Fix $A\in \Sigma$. Since $M(A)$ is relatively $\sigma(X,Y)$-compact,
it is $\sigma(X,Y)$-bounded and so
$$
	\tilde{M}(A):=\overline{{\rm conv}(M(A))}^{\sigma(X,Y)}
$$
is $\mu(X,Y)$-complete. Hence,  
Krein's theorem (see, e.g., \cite[7.1]{flo}) implies that $\tilde{M}(A)$ is $\sigma(X,Y)$-compact.

Note that for each $x^*\in Y$ we have $s(x^*,\tilde{M})=s(x^*,M)$, so $s(x^*,\tilde{M})$ is countably additive. Now, 
Theorem~\ref{theo:OT-2} applies to conclude that the set-valued map $\tilde{M}$ is a strong multimeasure and, therefore, $\tilde{M}$
is a multimeasure. 

Fix $x^*\in X^*$. Since $s(x^*,M)$ is finitely additive, both $s(x^*,\tilde{M})$ and $s(-x^*,\tilde{M})$ are countably additive and 
$$
	|s(x^*,M(A))| \leq |s(x^*,\tilde{M}(A))|+|s(-x^*,\tilde{M}(A))|
	\quad\text{for all $A\in \Sigma$} 
$$ 
(bear in mind Lemma~\ref{lem:inf}), we conclude that $s(x^*,M)$ is countably additive as well. This finishes the proof.
\end{proof}

A Banach space $X$ is said to have {\em property~$(\mathcal{E})$} (of Efremov) if for every convex bounded set $C\sub X^*$ 
and for every $x^*$ belonging to the weak$^*$-closure of~$C$ there is a sequence $(x_n^*)_n$ in~$C$ which weak$^*$-converges to~$x^*$.
This class of Banach spaces includes all weakly compactly generated ones, see \cite{avi-mar-rod,pli3,pli-yos-2} for further information.
In \cite[Section~3]{gui-mar-rod} it was shown that if $X$ has property~$(\mathcal{E})$
and $Y \sub X^*$ is a norming norm closed linear subspace, then $(X,\mu(X,Y))$ is complete. 

\begin{cor}\label{cor:Efremov}
Let $X$ be a Banach space having property~$(\mathcal{E})$ and let $Y \sub X^*$ be a norming norm closed linear subspace.
Let $M:\Sigma\to b(X)$ be a set-valued map defined on a $\sigma$-algebra~$\Sigma$ such that:
\begin{enumerate}
\item[(i)] $M(A)$ is relatively $\sigma(X,Y)$-compact for every $A\in \Sigma$;
\item[(ii)] $s(x^*,M)$ is countably additive for every $x^*\in Y$.
\end{enumerate}
Then the following statements hold: 
\begin{enumerate}
\item[(a)] $\bigcup_{A\in \Sigma}M(A)$ is relatively $\sigma(X,Y)$-compact.
\item[(b)] If $s(x^*,M)$ is finitely additive for all $x^*\in X^*$, then $M$ is a multimeasure.
\end{enumerate}
\end{cor}
\begin{proof}
The first part follows from Theorem~\ref{theo:range} and the comments preceeding the corollary. For the second part, 
since a Banach space having property~$(\mathcal{E})$ cannot contain subspaces isomorphic to~$\ell_\infty$
(see \cite[Section~3]{pli3}), we can apply the Diestel-Faires theorem and Proposition~\ref{pro:KS}.
\end{proof}

We are now ready to deal with the main result of this section.

\begin{proof}[Proof of Theorem~\ref{theo:DFJP}]
By Corollary~\ref{cor:weaklycompactrange}, the set $C:=\bigcup_{A\in \Sigma}M(A)$ is relatively weakly compact. 
The Davis-Figiel-Johnson-Pe{\l}czy\'nski factorization theorem (see, e.g., \cite[Theorem~13.22]{fab-ultimo}) ensures the existence of 
a reflexive Banach space~$Z$ and a one-to-one, bounded, linear operator $T:Z\to X$ 
such that $T(B_Z)\supseteq C$.

For each $A\in \Sigma$ we have $M(A) \sub T(B_Z)$, so $\hat{M}(A):=T^{-1}(M(A)) \sub B_Z$
and therefore $\hat{M}(A) \in cwk(Z)$. Let us check that the set-valued map $\hat{M}:\Sigma \to cwk(Z)$ satisfies 
the required properties. By definition, we have $T(\hat{M}(A))=M(A)$ for all $A\in \Sigma$. Note that $s(x^*\circ T,\hat{M})=s(x^*,M)$
is countably additive for all $x^*\in X^*$. Since $Y:=\{x^*\circ T:x^*\in X^*\}$ is a total linear subspace of~$Z^*$
(because $T$ is one-to-one) and $Z$ is reflexive, $Y$ is norm dense in~$Z^*$
and so it has the Orlicz-Thomas property (see, e.g., \cite[Lemma~3.1]{nyg-rod}).
From Corollary~\ref{cor:DF-2} we conclude that $\hat{M}$ is a multimeasure.
\end{proof}

\section{Appendix: selectors of multimeasures}

Given a locally convex Hausdorff space $E$, we denote by $E'$ its topological dual and we denote by 
$ck(E)$ the family of all convex compact non-empty subsets of~$E$. Given $\varphi \in E'$ and $K \in ck(E)$, 
we write $s(\varphi,K):=\sup\{\varphi(x) : x\in K\}$ (cf. Lemma~\ref{lem:inf}). If $M: \Sigma\to ck(E)$ is a set-valued map defined on a $\sigma$-algebra~$\Sigma$
and $\varphi\in E'$, then $s(\varphi,M): \Sigma \to \mathbb{R}$ is the map 
defined by $s(\varphi,M)(A):=s(\varphi,M(A))$ for all $A\in \Sigma$.
We write ${\rm ext}(C)$ to denote the set of all extreme points of a convex non-empty subset~$C$ of a linear space.

\begin{theo}\label{theo:ca-selector}
Let $E$ be a locally convex Hausdorff space and let $M:\Sigma \to ck(E)$ be a set-valued map defined on a $\sigma$-algebra~$\Sigma$
such that $s(\varphi,M)$ is countably additive for all $\varphi\in E'$.
Then there is a countably additive map $\nu:\Sigma \to E$ such that $\nu(A)\in {\rm ext}(M(A))$ for every $A\in \Sigma$.
\end{theo}

Lemma~\ref{lem:fa-selector} below (cf. \cite[Theorem~8.6]{dre}) is a key step towards the previous theorem. First, we need an elementary observation.

\begin{lem}\label{lem:extreme1}
Let $E$ be a linear space, let $C_1$ and $C_2$ be convex non-empty subsets of~$E$ and let $x\in {\rm ext}(C_1+C_2)$. Then there exist unique $x_1\in C_1$
and $x_2\in C_2$ such that $x=x_1+x_2$. Moreover, $x_1\in {\rm ext}(C_1)$ and $x_2\in {\rm ext}(C_2)$.
\end{lem}
\begin{proof}
Take $x_1,x'_1\in C_1$ and $x_2,x'_2\in C_2$ such that $x=x_1+x_2=x'_1+x'_2$.  Then
$$
	x=\frac{1}{2}(x_1+x'_2)+\frac{1}{2}(x'_1+x_2)
$$
with $x_1+x'_2\in C_1+C_2$ and $x'_1+x_2\in C_1+C_2$. Since $x \in {\rm ext}(C_1+C_2)$, we get
$$
	x_1+x'_2=x'_1+x_2=x=x_1+x_2=x'_1+x'_2,
$$
so $x_1=x'_1$ and $x_2=x'_2$, as required. 

To check that $x_1 \in {\rm ext}(C_1)$, take $u,v\in C_1$ such that $x_1=\frac{1}{2}(u+v)$. Then
$$
	x=\frac{1}{2}(u+x_2)+\frac{1}{2}(v+x_2)
$$
with $u+x_2 \in C_1+C_2$ and $v+x_2\in C_1+C_2$. Since $x \in {\rm ext}(C_1+C_2)$, we have
$$
	u+x_2=v+x_2
$$
and therefore $u=v$. It follows that $x_1\in {\rm ext}(C_1)$. Analogously, $x_2\in {\rm ext}(C_2)$.
\end{proof}

\begin{lem}\label{lem:fa-selector}
Let $E$ be a locally convex Hausdorff space, let $\mathcal{A}$ be an algebra of subsets of a set~$\Omega$ and let $M:\mathcal{A} \to ck(E)$
be a finitely additive set-valued map, that is, $M(A\cup B)=M(A)+M(B)$ whenever $A,B\in \mathcal{A}$ are disjoint.
Then there is a finitely additive map $\nu:\mathcal{A} \to E$ such that $\nu(A)\in {\rm ext}(M(A))$ for every $A\in \mathcal{A}$.
\end{lem}
\begin{proof}
Since $M(\Omega)$ is convex, compact and non-empty, 
the set ${\rm ext}(M(\Omega))$ is non-empty, by the Krein-Milman theorem (see, e.g., \cite[Theorem~3.65]{fab-ultimo}). Let us fix $x\in {\rm ext} (M(\Omega))$.
For each $A\in \mathcal{A}$ we have $M(\Omega)=M(A)+M(\Omega \setminus \mathcal{A})$
and so there exist $\nu(A) \in {\rm ext}(M(A))$ and $\nu(\Omega \setminus A) \in {\rm ext}(M(\Omega\setminus A))$
such that
$$
	x=\nu(A)+\nu(\Omega \setminus A)
$$
(by Lemma~\ref{lem:extreme1}).

To check that the map $\nu:\mathcal{A}\to E$ is finitely additive, take disjoint $A_1,A_2\in \mathcal{A}$. Then
$$
	\nu(A_1\cup A_2)\in M(A_1\cup A_2)=M(A_1)+M(A_2)
$$
and so we can write $\nu(A_1\cup A_2)=x_1+x_2$ for some $x_1\in M(A_1)$ and $x_2\in M(A_2)$. Observe that
by the definition of~$\nu$ we have
$$
	x=\nu(A_1\cup A_2) +\nu(\Omega \setminus (A_1\cup A_2))=x_1+(x_2+\nu(\Omega \setminus (A_1\cup A_2))),
$$
where $x_1\in M(A_1)$ and
$$
	x_2+\nu(\Omega \setminus (A_1\cup A_2))\in M(A_2)+M(\Omega \setminus (A_1\cup A_2))=M(\Omega \setminus A_1).
$$
By the uniqueness part of Lemma~\ref{lem:extreme1}, $x_1=\nu(A_1)$.
A similar argument yields $x_2=\nu(A_2)$. Hence, $\nu(A_1\cup A_2)=\nu(A_1)+\nu(A_2)$, as required.
\end{proof}

\begin{proof}[Proof of Theorem~\ref{theo:ca-selector}]
Note first that $M$ is finitely additive. Indeed, let $A,B \in \Sigma$ be disjoint. Then for each $\varphi\in E'$
the finite additivity of $s(\varphi,M)$ yields
$$
	s(\varphi,M(A\cup B))=s(\varphi,M(A))+s(\varphi,M(B))=s(\varphi,M(A)+M(B)).
$$
Since both $M(A\cup B)$ and $M(A)+M(B)$ are convex and closed, the Hahn-Banach theorem ensures that
$M(A\cup B)=M(A)+M(B)$, as required.

By Lemma~\ref{lem:fa-selector},
there is a finitely additive map $\nu:\Sigma \to E$ in such a way that $\nu(A)\in {\rm ext}(M(A))$ for every $A\in \Sigma$.

Fix $\varphi\in E'$. We claim that the composition $\varphi \circ \nu:\Sigma \to \mathbb{R}$ is countably additive. Indeed,
just observe that $\varphi\circ \nu$ is finitely additive, we have
$$
	|\varphi(\nu(A))| \leq |s(\varphi ,M(A))|+|s(-\varphi,M(A))|
	\quad
	\text{for all $A\in \Sigma$}
$$
(by Lemma~\ref{lem:inf}) and both $s(\varphi,M)$ and $s(-\varphi,M)$ are countably additive. 

By the Orlicz-Pettis theorem in the setting of locally convex spaces (see, e.g., \cite[p.~308]{jar}),
$\nu$ is countably additive. The proof is finished.
\end{proof}

\subsection*{Acknowledgements}
The research was supported by grants PID2021-122126NB-C32 
(funded by MCIN/AEI/10.13039/501100011033 and ``ERDF A way of making Europe'', EU) and 
21955/PI/22 (funded by {\em Fundaci\'on S\'eneca - ACyT Regi\'{o}n de Murcia}).


\bibliographystyle{amsplain}

\end{document}